\documentclass{article}

\usepackage{arxiv}

\usepackage[utf8]{inputenc} 
\usepackage[T1]{fontenc}    
\usepackage{hyperref}       
\usepackage{url}            
\usepackage{booktabs}       
\usepackage{amsfonts}       
\usepackage{nicefrac}       
\usepackage{microtype}      
\usepackage{lipsum}
\usepackage{graphicx}
\graphicspath{ {./images/} }

\usepackage{bm}
\usepackage{amsthm,amsfonts,epsfig}
\usepackage{amsmath}
\usepackage[T1]{fontenc}
\usepackage{epsfig}
\usepackage{graphicx}
\usepackage{graphics}
\usepackage{bbm}
\usepackage{enumitem}
\usepackage{subfigure}
\usepackage{placeins}
\usepackage{amssymb,srcltx,amsthm,epsfig}
\usepackage{color}

\newcommand{\bs}{\mathbf{s}}
\newcommand{\ii}{i=1,\ldots,p}

\newcommand{\by}{{\bf y}}
\newtheorem{teorema}{Theorem}

\newtheorem{proposition}{Proposition}
\newtheorem{corollary}{Corollary}

\newtheorem{remark}{Remark}

\newcommand{\x}{\mathbf{x}}
\newcommand{\q}{\mathbf{q}}

\newcommand{\ep}{\mathbf{e}}
\newcommand{\bd}{\mathbf{d}}
\newcommand{\yp}{\mathbf{y}}
\newcommand{\Y}{\mathbf{Y}}

\newcommand{\zero}{\mathbf{0}}
\newcommand{\bI}{\mathbf{I}}
\newcommand{\tb}{\mathbf{t}}
\newcommand{\Hb}{\mathbf{H}}
\newcommand{\Sb}{\mathbf{S}}

\newcommand{\W}{\mathbf{W}}

\newcommand{\Z}{\mathbf{Z}}

\newcommand{\zp}{\mathbf{z}}
\newcommand{\A}{\mathbf{A}}
\newcommand{\ap}{\mathbf{a}}
\newcommand{\bp}{\mathbf{b}}
\newcommand{\X}{\mathbf{X}}
\newcommand{\R}{\mathbf{R}}

\newcommand{\D}{\mathbf{D}}

\newcommand{\FF}{\mathcal{F}}
\newcommand{\LL}{\mathcal{L}}
\newcommand{\II}{\mathcal{I}}

\newcommand{\B}{\mathbf{B}}
\newcommand{\bx}{\mathbf{x}}
\newcommand{\xp}{\mathbf{x}}

\newcommand{\ESN}{\textrm{ESN}}

\newcommand{\TESN}{\textrm{TESN}}

\newcommand{\y}{\mathbf{y}}

\newcommand{\dr}[1]{{\mathrm d}#1}

\newcommand{\bmu}{{ \boldsymbol{\mu}}}

\newcommand{\bSigma}{\boldsymbol{\Sigma}}

\newcommand{\bLambda}{\boldsymbol{\Lambda}}

\newcommand{\btheta}{\boldsymbol{\theta}}

\newcommand{\bxi}{\boldsymbol{\xi}}

\newcommand{\bvarphi}{\boldsymbol{\varphi}}

\newcommand{\bDelta}{\boldsymbol{\Delta}}
\newcommand{\bdelta}{\boldsymbol{\delta}}
\newcommand{\bPsi}{\boldsymbol{\Psi}}

\newcommand{\btau}{\boldsymbol{\tau}}

\newcommand{\bOmega}{\boldsymbol{\Omega}}

\newcommand{\bGamma}{\boldsymbol{\Gamma}}

\newcommand{\blambda}{\boldsymbol{\lambda}}
\newcommand{\bkappa}{\boldsymbol{\kappa}}
\newcommand{\binfty}{\boldsymbol{\infty}}
\newcommand{\binf}{\boldsymbol{\infty}}

\newcommand{\TN}{\textrm{TN}}

\newcommand{\tautil}{\tilde\tau}
\newcommand{\CG}[1]{\textcolor{black}{#1}}
\newcommand{\CGT}[1]{\textcolor{black}{#1}}

\newcommand{\EE}{\mathbbm{E}}

\title{On moments of folded and doubly truncated multivariate extended skew-normal distributions}

\author{
Christian E. Galarza \\
Departamento de Estad\'{\i}stica\\
Escuela Superior Politecnica del Litoral\\
Guayaquil, Ecuador \\
\texttt{chedgala@espol.edu.ec} \\
\And
Larissa A. Matos \\
Departamento de Estat\'{\i}stica\\
Universidade Estadual de Campinas\\
Campinas, Brazil\\
\texttt{larissam@unicamp.br} \\
\And
Dipak K. Dey  \\
Department of Statistics\\
University of Connecticut\\
Storrs CT  06269, U.S.A. \\
\texttt{dipak.dey@uconn.edu} \\
\And
Victor H. Lachos \\
Department of Statistics\\
University of Connecticut\\
Storrs CT  06269, U.S.A. \\
\texttt{hlachos@uconn.edu} \\
}

\begin{document}
\maketitle
\begin{abstract}
{This paper develops recurrence relations for integrals that relate the density of multivariate extended skew-normal (ESN) distribution, including the well-known skew-normal  (SN) distribution introduced by \cite{AzzaliniDV1996} and the popular multivariate normal distribution. These recursions {offer} a fast computation of arbitrary order product moments of the multivariate truncated extended skew-normal and multivariate folded extended skew-normal distributions with the product moments as a byproduct.} In addition to the recurrence approach, we realized that any arbitrary moment of the truncated multivariate extended skew-normal distribution can be computed using a corresponding moment of a truncated multivariate normal distribution, pointing the way to a faster algorithm since a less number of integrals is required for its computation which result much simpler to evaluate. Since there are several methods available to calculate the first two moments of a multivariate truncated normal distribution, we propose an optimized method that offers a better performance in terms of time and accuracy, in addition to consider extreme cases in which other methods fail. The {\sf R MomTrunc} package provides these new efficient methods for practitioners.
\end{abstract}

\section{Introduction} \label{sec:intro}

{Many applications on simulations or experimental studies, the researches} often generate a large number of datasets with restricted values to fixed intervals. For example, variables such as pH, grades, viral load in HIV studies and humidity in environmental studies, have upper and lower bounds {due to detection limits}, and the support of their densities is restricted to some {given intervals}. Thus, the {need to study truncated distributions along with their} properties naturally arises. In this context, there has been a growing interest in evaluating the moments of truncated distributions. These variables are also often skewed, departing from the traditional assumption of using symmetric distributions.
For instance, \cite{Tallis1961} provided the formulas for the first two moments of truncated multivariate normal (TN) distributions. \cite{lien1985moments} gave the expressions for the moments of truncated bivariate log-normal distributions with applications to test the Houthakker effect (\cite{houthakker1959scope}) in future markets. \cite{jawitz2004moments} derived the truncated moments of several continuous univariate distributions commonly applied to hydrologic problems. \cite{Kim2008} provided analytical formulas for moments of the truncated univariate Student-t distribution in a recursive form. \cite{flecher2010truncated} obtained expressions for the moments of truncated univariate skew-normal {distributions (\cite{Azzalini1985}}) and applied the results to model {the} relative humidity data. \cite{genc} {studied} the moments of a doubly truncated member of the symmetrical class of univariate normal/independent distributions {and their applications to the} actuarial data. \cite{lin2011some} {presented} a general formula {based on the slice sampling algorithm to approximate the first two moments of the truncated multivariate Student-$t$ (TT) distribution under the double truncation.} \cite{arismendi2013multivariate} provided explicit expressions for computing arbitrary order product moments of the TN distribution by using the moment generating function (MGF). However, the {calculation of this approach relies on differentiation of the MGF and can be somewhat} time consuming.

Instead of differentiating the {MGF of the TN} distribution, \cite{kan2017moments} recently presented recurrence relations for integrals that are directly related to the density of the multivariate normal distribution for computing arbitrary order product moments of the TN distribution. These recursions {offer} a fast computation of the moments of folded normal (FN) and TN distributions, which require evaluating {$p$}-dimensional integrals that involve the Normal (N) density. Explicit expressions for some low order moments of FN and TN distributions are presented in a clever way, although some proposals to {calculate} the moments of the univariate truncated skew-normal distribution (\cite{flecher2010truncated}) and truncated univariate skew-normal/independent distribution (\cite{flecher2010truncated})  has recently been published{. So} far, to the best of our knowledge, there has not been attempt on studying neither moments nor product moments of the multivariate folded extended skew-normal (FESN) and truncated multivariate extended skew-normal (TESN) distributions. Moreover, our proposed methods allow to compute, as a by-product, the product moments of folded and truncated  distributions, of the N (\cite{kan2017moments}), SN (\cite{AzzaliniDV1996}), and their respective univariate versions. The proposed algorithm and methods are implemented in the new R package ``MomTrunc''.

The rest of {this} paper is organized as follows. In Section \ref{FM_SNLMM} we briefly discuss some preliminary results related to the multivariate SN, ESN and TESN distributions and some of its key properties.  The section 3 presents a recurrence formula {of} an integral {to be applied in the essential evaluation} of moments of the TESN distribution as well as explicit expressions for the first two {moments} of the TESN and TN distributions. A direct relation between the moments of the TESN and TN distribution is also presented which is used to improved the proposed methods. In section 4, by means of approximations, we propose strategies to circumvent some numerical problems that arise on limiting distributions and extreme cases. We compare our proposal with others popular methods of the literature in Section 5. Finally, Section \ref{sec:folded:ESN} is devoted to the moments of the FESN distribution, several related results are discussed. Explicit expressions are presented for high order moments for the univariate case and the mean vector and variance-covariance matrix of the multivariate FESN distribution. Finally, some concluding remarks are presented in Section \ref{Conclusion}.

\section{Preliminaries}\label{FM_SNLMM}

We start our exposition by {defining} some notation and presenting the basic concepts {which are used throughout the development of our theory}. {As is usual in probability theory and its applications, we denote} a random variable by an {upper-case} letter and its realization by the corresponding lower case {and} use boldface letters {for vectors and matrices.} {Let $\mathbf{I}_p$ and $\mathbf{J}_p$ represent} an identity matrix and a matrix of ones, respectively, both of dimension $p\times p$, {$\A^{\top}$ be the transpose of $\A$, and $|\X|=(|X_1|.\ldots, |X_p|)^{\top}$ mean} the absolute value of each component of the vector $\X$.  For multiple {integrals, we use} the {shorthand} notation
$$\int_{\ap}^{\bp}f(\x)d\x=\int_{a_1}^{b_1}\ldots\int_{a_p}^{b_p}f(x_1,\ldots,x_p)dx_1\ldots x_p,$$ 
where  $\mathbf{a}=(a_1,\ldots,a_p)^\top$  and $\mathbf{b}=(b_1,\ldots,b_p)^\top$.

\subsection{The multivariate skew-normal distribution}

In this subsection we present the skew-normal distribution and some of its properties. We say that a $p\times 1$ random vector ${\Y}$ follows a multivariate SN distribution with $p\times 1$ location vector $\bmu$, $p\times p$ positive definite dispersion matrix $\bSigma$ and $p\times 1$ skewness parameter vector, and we write $\textbf{Y}\sim \textrm{SN}_p(\bmu,\bSigma,\blambda),$ if its joint probability density function (pdf) is given by
\begin{equation}\label{denSN}
SN_p(\mathbf{y};\bmu,\bSigma,\blambda)= 2{\phi_p(\mathbf{y};\bmu,\bSigma)
\Phi_1(\blambda^{\top}\bSigma^{-1/2}(\mathbf{y}-\bmu))},
\end{equation}
{where $\phi_p(\cdot;\bmu,\bSigma)$ represents the probability density distribution (pdf) of a $p$-variate normal distribution with vector mean $\bmu$ and variance-covariance matrix $\bSigma$,} and $\Phi_1(\cdot)$  stands for the cumulative distribution function (cdf) of a standard univariate normal distribution. If  $\blambda=\bf 0$  then  (\ref{denSN}) reduces to the symmetric $\textrm{N}_p(\bmu,\bSigma)$  pdf.  Except by a straightforward difference in the parametrization considered in (\ref{denSN}), this model corresponds to the one introduced by \cite{AzzaliniDV1996}, whose properties  were extensively studied in \cite{AzzaliniC1999} (see
also, \cite{ArellanoG2005}).

\begin{proposition}[\textbf{\emph{cdf} of the SN}]\label{propdenscond} If $\Y\sim
SN_{p}(\bmu,\bSigma,\blambda)$, then for any $\yp\in
\mathbb{\mathbb{R}}^p$
\begin{equation*}\label{denscond}
F_{\Y}(\yp)=P(\Y\leq\yp)=2{\Phi_{p+1}\hspace{-0.2mm}\big((\mathbf{y}^
{\scriptscriptstyle\top},0)^{\top};\bmu^*,\bOmega
\big)},\,\,
\end{equation*}
where $\bmu^*=(\bmu^{\scriptscriptstyle\top},0)^{\top}$ and $\bOmega=\left(\begin{array}{cc}
\bSigma& -\bDelta \\
-\bDelta^{\top} & 1
\end{array}
\right),$ with $\bDelta=\bSigma^{1/2}\blambda/{(1+\blambda^{\top}\blambda)^{1/2}}.$
\end{proposition}

It is worth mentioning that the multivariate skew-normal distribution is not closed over marginalization and conditioning. Next, we present its extended version which holds these properties, called, the multivariate ESN distribution.

\subsection{The extended multivariate skew-normal distribution}

We say that a $p\times 1$ random vector ${\Y}$ follows a ESN distribution with $p\times 1$ location vector $\bmu$, $p\times p$ positive definite dispersion matrix $\bSigma$, a $p\times 1$ skewness parameter vector, and shift parameter $\tau\in\mathbb{R}$, denoted by $\textbf{Y}\sim \textrm{ESN}_p(\bmu,\bSigma,\blambda,\tau),$ if its pdf is given by
\begin{equation}\label{denESN}
ESN_p(\mathbf{y};\bmu,\bSigma,\blambda,\tau)= \xi^{-1}{\phi_p(\mathbf{y};\bmu,\bSigma)
\Phi_1(\tau+\blambda^{\top}\bSigma^{-1/2}(\mathbf{y}-\bmu))},
\end{equation}
with $\xi=\Phi_1(\tau/(1+\blambda^{\top}\blambda)^{1/2})$. Note that when $\tau=0$, we retrieve the skew-normal distribution defined in (\ref{denSN}), that is, $ ESN_p(\mathbf{y};\bmu,\bSigma,\blambda,0)=SN_p(\mathbf{y};\bmu,\bSigma,\blambda)$.
Here, we used a slightly different parametrization of the ESN distribution than the one given in \cite{ArellanoA2006} and \cite{arellano2010multivariate}.
Futhermore, \cite{arellano2010multivariate} deals with the multivariate extended skew-t (EST) distribution, in which the ESN is a particular case when the degrees of freedom $\nu$ goes to infinity. From this last work, it is straightforward to see that
$$  ESN_p(\mathbf{y};\bmu,\bSigma,\blambda,\tau) {\longrightarrow} \phi_p(\mathbf{y};\bmu,\bSigma),\,\,{\text as }\,\,\,\tau\rightarrow +\infty.$$
Also, letting $\Z = \bSigma^{-1/2}(\Y-\bmu)$, it follows that $\Z \sim \ESN_{p}(\zero,\mathbf{I},\blambda,\tau)$, with mean vector and variance-covariance matrix
\begin{equation*}
\mathbb{E}[\Z] = \eta\blambda \quad\quad \text{and}\quad\quad
\mathrm{cov}[\Z] = \mathbf{I}_p - \mathbb{E}[\Z]\left(\mathbb{E}[\Z] - \frac{\tau}{1+\blambda^{\top}\blambda}\blambda\right)^{\top},
\end{equation*}
with $\eta = {\phi_1(\tau;0,1+\blambda^{\top}\blambda)}/\xi$. Then, the mean vector and variance-covariance matrix of $\Y$ can be easily computed as $\mathbb{E}[\Y] = \bmu + \bSigma^{1/2}\mathbb{E}[\Z]$ and $\mathrm{cov}[\Y]=\bSigma^{1/2}\mathrm{cov}[\Z]\bSigma^{1/2}$.

The following  propositions are crucial to develop our methods. The proofs can be found in the Appendix A.

\begin{proposition}[\textit{Marginal and conditional distribution of the ESN}]\label{proposition2}
Let  $\Y\sim ESN_{p}(\bmu,\bSigma,\blambda,\tau)$ and $\Y$ is
partitioned as $\Y=(\Y^{\top}_1,\Y^{\top}_2)^{\top}$ of
dimensions $p_1$ and $p_2$ ($p_1+p_2=p$), respectively. Let
$$\bSigma=\left(\begin{array}{cc}
\bSigma_{11} & \bSigma_{12} \\
\bSigma_{21} & \bSigma_{22}
\end{array}
\right),\ \ \bmu=(\bmu^{\top}_1,\bmu^{\top}_2)^{\top}, \ \ \blambda=(\blambda^{\top}_1,\blambda^{\top}_2)^{\top}\quad\text{and}\quad \bvarphi=(\bvarphi^{\top}_1,\bvarphi^{\top}_2)^{\top}$$ be the
corresponding partitions of $\bSigma$, $\bmu$, $\blambda$ and $\bvarphi=\bSigma^{-1/2}\blambda$. Then,
\begin{equation*}
\Y_1 \sim ESN_{p_1}(\bmu_1,\bSigma_{11}, c_{12}\bSigma_{11}^{1/2}\tilde{\bvarphi}_1,c_{12}\tau) \ \mbox{ and } \  \Y_2|\Y_1=\y_1\sim ESN_{p_2}(\bmu_{2.1},\bSigma_{22.1},\bSigma^{1/2}_{22.1}\bvarphi_2, \tau_{2.1}),
\end{equation*}
where $c_{12}=(1+\bvarphi^{\top}_2\bSigma_{22.1}\bvarphi_2)^{-1/2}$, $\tilde{\bvarphi}_1=\bvarphi_1+\bSigma_{11}^{-1}\bSigma_{12}\bvarphi_2$, $\bSigma_{22.1}=\bSigma_{22}-\bSigma_{21}\bSigma^{-1}_{11}\bSigma_{12}$,  $\bmu_{2.1}=\bmu_2+\bSigma_{21}\bSigma^{-1}_{11}(\by_1-\bmu_1)$ and $\tau_{2.1}=\tau+\tilde{\bvarphi}^{\top}_1(\by_1-\bmu_1)$.
\end{proposition}

\begin{proposition}[\textit{Stochastic representation of the ESN}]\label{stochastic}  Let $\X = ({\X_1}^\top,X_2)^\top
\sim N_{p+1}(\bmu^*,\bOmega)$. If
$\Y \stackrel{d}{=} (\X_1|X_2<\tilde{\tau}),$
it follows that $\Y \sim \ESN_{p}(\bmu,\bSigma,\blambda,\tau)$, with $\bmu^*$ and $\bOmega$ as defined in Proposition \ref{propdenscond}, and $\tilde{\tau}=\tau/(1+\blambda^{\top}\blambda)^{1/2}$.
\end{proposition}

The stochastic representation above can be derived from Proposition 1 in \cite{arellano2010multivariate}.

\begin{proposition}[\textit{\emph{cdf} of the ESN}]\label{propdenscond3} If $\Y\sim
\ESN_{p}(\bmu,\bSigma,\blambda,\tau)$, then for any $\yp\in
\mathbb{\mathbb{R}}^p$
\begin{equation*}
F_{\Y}(\yp)=P(\Y\leq\yp)=
{\xi^{-1}}
\Phi_{p+1}\hspace{-0.2mm}\big((
\mathbf{y}^{\scriptscriptstyle\top},\tilde\tau)^{\top};
\bmu^*,\bOmega
\big).
\end{equation*}
\end{proposition}

Proof is direct from Proposition \ref{stochastic} {by noting that $\xi = P(X_2<\tilde\tau)$}. Hereinafter, for $\Y\sim
\ESN_{p}(\bmu,\bSigma,\blambda,\tau)$, we will denote to its \emph{cdf} as $F_{\Y}(\yp)\equiv \tilde{\Phi}_p(\yp;\bmu,\bSigma,\blambda,\tau)$ for simplicity.\\

Let $\mathbb{A}$ be a Borel set in $\mathbb{R}^p$. We say that the random vector $\Y$ has a truncated extended skew-normal distribution on $\mathbb{A}$ when $\Y$ has the same distribution as $\Y | (\Y \in \mathbb{A})$. In this case, the pdf of $\Y$ is given by
$$
f(\y\mid\bmu,\bSigma,\nu;\mathbb{A})=\displaystyle\frac{ESN_p(\y;\bmu,\bSigma,\blambda,\tau)}{P(\Y \in  \mathbb{A})}\mathbf{1}_{\mathbb{A}}(\y),
$$
where  $\mathbf{1}_{ \mathbb{A}}$ is the indicator function of $ \mathbb{A}$. We use the notation $\Y \sim  {\TESN}_{p}(\bmu,\bSigma,\blambda,\tau;\mathbb{A})$. If $\mathbb{A}$ has the form
\begin{equation*} \label{hyper}
\mathbb{A} = \{(x_1,\ldots,x_p)\in \mathbb{R}^p:\,\,\, a_1\leq x_1 \leq b_1,\ldots, a_p\leq x_p \leq b_p \}=\{\mathbf{x}\in\mathbb{R}^p:\mathbf{a}\leq\mathbf{x}\leq\mathbf{b}\},
\end{equation*}
then we use the notation $\{\Y \in  \mathbb{A}\}=\{\ap \leq \Y \leq \mathbf{b}\}$, where  $\mathbf{a}=(a_1,\ldots,a_p)^\top$ and $\mathbf{b}=(b_1,\ldots,b_p)^\top$. Here, we say that  the distribution of $\Y$ is doubly truncated. Analogously, we define $\{\Y \geq \mathbf{a}\}$ and $\{\Y \leq \mathbf{b}\}$. Thus, we say that the distribution of $\Y$ is truncated from below and truncated from above, respectively. For convenience, we also use the notation $\Y \sim  {\TESN}_{p}(\bmu,\bSigma,\blambda,\tau;[\ap,\bp])$.

\section{On moments of the doubly truncated multivariate ESN distribution}

\subsection{A recurrence relation}

For two $p$-{dimensional vectors} $\bx=(x_1,\ldots,x_p)^{\top}$ and $\bkappa=(k_1,\ldots,k_p)^{\top}$,  let $\bx^{\bkappa}$ \ stand \ for \  {$(x_1^{\kappa_1},\ldots,x_p^{\kappa_p})$, and let} $\ap_{(i)}$ be a vector $\ap$ with its $i$th element {being} removed. For a matrix \CGT{$\A$, we let $\A_{i(j)}$ stand for the $i$th row of $\A$ with its $j$th element {being} removed. Similarly, $\A_{(i,j)}$ stands for the matrix $\A$} with its $i$th row and $j$th columns {being} removed. {Besides, let} $\ep_i$ denote {a $p\times 1$ vector} with its $i$th element {equaling} one and zero otherwise. Let
$$\mathcal{L}_p(\ap,\bp;\bmu,\bSigma,\blambda,\tau)=
\int_{\ap}^{\bp}{ESN_p(\mathbf{x};\bmu,\bSigma,\blambda,\tau)
}\dr{\xp}.$$
We are interested in evaluating the integral
\begin{equation}\label{Fk_ESN}
\FF^p_{\bkappa}(\ap,\bp;\bmu,\bSigma,\blambda,\tau)=
\int_{\ap}^{\bp}{\bx^{\bkappa}{ESN_p(\mathbf{x};\bmu,\bSigma,\blambda,\tau)
}}\dr{\xp}.
\end{equation}
The boundary condition is obviously $\FF^p_{0}(\ap,\bp;\bmu,\bSigma,\blambda,\tau)=\LL_p(\ap,\bp;\bmu,\bSigma,\blambda,\tau)$. When $\blambda=\zero$ and $\tau=0$, we recover the multivariate normal case, and then
\begin{equation}\label{Fk_ESN1}
\FF^p_{\bkappa}(\ap,\bp;\bmu,\bSigma,\zero,0)\equiv F^p_{\bkappa}(\ap,\bp;\bmu,\bSigma)= \int_{\ap}^{\bp}{\bx^{\bkappa}{{\phi}_p}(\mathbf{x};\bmu,\bSigma)\textrm{d}\bx,}
\end{equation}
with boundary condition
\begin{equation}\label{Fk_ESN2}
\LL_p(\ap,\bp;\bmu,\bSigma,\zero,0) \equiv L_p(\ap,\bp;\bmu,\bSigma) = \int_{\ap}^{\bp}{{{\phi}_p}(\mathbf{x};\bmu,\bSigma)\textrm{d}\bx}.
\end{equation}
Note that we use calligraphic style for the integrals of interest $\FF^p_{\bkappa}$ and $\LL_p$ when we work with the skewed version. In both expressions \eqref{Fk_ESN1} and \eqref{Fk_ESN2}, for the normal case, we are using compatible notation with the one used by \cite{kan2017moments}.

\subsubsection{Univariate case}\label{univariate_rec}

When $p=1$, it is straightforward to use integration by parts to show that
\begin{eqnarray*}
\FF^1_0(a,b;\mu,\sigma^2,\lambda,\tau)&=&\xi^{-1}\left[\Phi_2\left((b-\mu,{\tau})^{\scriptscriptstyle\top};\mathbf{0},\bOmega
\right)-\Phi_2\left((a-\mu,{\tau})^{\scriptscriptstyle\top};\mathbf{0},\bOmega
\right)\right],\\
\FF^1_{k+1}(a,b;\mu,\sigma^2,\lambda,\tau)&=&\mu \FF^1_{k}(a,b;\mu,\sigma^2,\lambda,\tau)+k\sigma^2 \FF^1_{k-1}(a,b;\mu,\sigma^2,\lambda,\tau)\\ 
&&+\sigma^2\big(a^k ESN_1(a;\mu,\sigma^2,\lambda,\tau)
-b^k ESN_1(b;\mu,\sigma^2,\lambda,\tau)\big)\\ 
&&+\lambda\sigma \eta F^1_k(a,b;\mu-\mu_b,\gamma^2);\,\,\,\,\mbox{for } k \geq 0,
\end{eqnarray*}
where 
$\bOmega=\left(\begin{array}{cc}
\sigma^2& -\sigma\psi \\
-\sigma\psi & 1
\end{array}
\right)$,  $\psi=\lambda/\sqrt{1+\lambda^2}$, $\mu_b=\displaystyle\frac{1}{\sigma}\lambda\tau\gamma^2$ and $\gamma=\sigma/\sqrt{1+\lambda^2}$.

When $p>1$, we need a similar recurrence relation in order to compute $ \FF^p_{\bkappa}(\ap,\bp,\bmu,\bSigma,\blambda,\tau)$ which is presented in the next theorem.

\subsubsection{Multivariate case}

\begin{teorema}\label{ESN_theo}
For $p\geq 1$	and $i=1,\ldots,p$,
\begin{eqnarray}\label{ESN_theo_eq1}
\FF ^p_{\bkappa+\ep_i}(\ap,\bp;\bmu, \bSigma,\blambda, \tau )=\mu_i  \FF^p_{\bkappa}(\ap,\bp;\bmu, \bSigma,\blambda, \tau )+\delta_iF^p_{\bkappa}(\ap,\bp;\bmu-\bmu_b,\bGamma)+ \ep^{{\top}}_i\bSigma \bd_{\bkappa},
\end{eqnarray}
where 
$\bdelta=(\delta_1,\ldots, \delta_p)^{\top}=\eta\bSigma^{1/2}\blambda$,
$\bmu_b=\tautil\bDelta$,
$\bGamma= \bSigma - \bDelta\bDelta^\top$
and $\bd_{\bkappa}$	is a $p$-vector with $j$th element
\begin{eqnarray}\label{recur1}
d_{\bkappa,j}&=&k_j \FF ^p_{\bkappa-\ep_j}(\ap,\bp,\bmu, \bSigma,\blambda, \tau )\nonumber\\
&&+a_j^{k_j} \ESN_1(a_j;\mu_j, \sigma^2_j, c_j\sigma_j\tilde{\varphi}_j,c_j\tau) \FF^{p-1}_{\bkappa_{(j)}}(\ap_{(j)},\bp_{(j)};\tilde \bmu^{\ap}_j,\tilde\bSigma_j,\tilde\bSigma^{\scriptscriptstyle 1/2}_j\bvarphi_{(j)},\tilde \tau^{\ap}_j)\nonumber\\
&&-b_j^{k_j} \ESN_1(b_j;\mu_j, \sigma^2_j, c_j\sigma_j\tilde{\varphi}_j,c_j\tau) \FF^{p-1}_{\bkappa_{(j)}}(\ap_{(j)},\bp_{(j)};\tilde \bmu^{\bp}_j,\tilde\bSigma_j,\tilde\bSigma^{\scriptscriptstyle 1/2}_j\bvarphi_{(j)},\tilde \tau^{\bp}_j),
\end{eqnarray}
\noindent where
{\small $$
\tilde\bmu_j^\ap = \bmu_{(j)}+\bSigma_{(j),j}\frac{a_j-\mu_j}{\sigma_j^2}, \,\,
\tilde\bmu_j^\bp = \bmu_{(j)}+\bSigma_{(j),j}\frac{b_j-\mu_j}{\sigma_j^2 },
\,\,
\tilde\varphi_j=\varphi_j+\frac{1}{\sigma^2_j}\bSigma_{j(j)}\bvarphi_{(j)},
\,\,
\tilde\bSigma_j = \bSigma_{(j),(j)}-\frac{1}{\sigma_j^2}\bSigma_{(j),j}\bSigma_{j,(j)},
$$
$$c_j=\frac{1}{(1+\bvarphi^{\top}_{(j)}\tilde\bSigma_j\bvarphi_{(j)})^{1/2}},
\qquad
\tilde\tau^{\ap}_j=\tau+\tilde\varphi_j(a_j-\mu_j),
\qquad\text{and}\qquad
\tilde\tau^{\bp}_j=\tau+\tilde\varphi_j(b_j-\mu_j).
$$}
\end{teorema}

\medskip
\begin{proof}
Let $\X = ({\X_1}^\top,X_2)^\top
\sim N_{p+1}(\bmu^*,\bOmega)$ as in Proposition 2. From the conditional distribution of a multivariate normal, it is straightforward to show that $\X_1|X_2 \sim N_p(\bmu - X_2\bDelta,\bGamma)$ and
$X_2|\X_1 \sim N_1(-\bDelta^\top\bSigma^{-1}(\X_1-\bmu),1 - \bDelta^\top\bSigma^{-1}\bDelta)$. Then it holds that
\begin{align}
f_{X_2|\X_1}(\tautil|\xp)f_{\X_1}(\xp) &= f_{X_2}(\tautil)f_{\X_1|X_2}(\xp|\tautil)
\nonumber
\\
\phi_1(\tautil;-\bDelta^\top\bSigma^{-1}(\xp-\bmu),1 - \bDelta^\top\bSigma^{-1}\bDelta)
\phi_p(\xp;\bmu,\bSigma)
&= \phi_1(\tautil)
\phi_p(\xp;\bmu - \tautil\bDelta,\bGamma)\nonumber\\
\sqrt{1+\blambda^{\scriptscriptstyle\top}\blambda}\times
\phi_1(\tau+\blambda^\top\bSigma^{-1/2}(\xp-\bmu))
\phi_p(\xp;\bmu,\bSigma)
&= \phi_1(\tautil)
\phi_p(\xp;\bmu - \bmu_b,\bGamma)\nonumber\\
\phi_1(\tau+\blambda^\top\bSigma^{-1/2}(\xp-\bmu))
\phi_p(\xp;\bmu,\bSigma)
&= \phi_1(\tau;0,1+\blambda^{\scriptscriptstyle\top}\blambda)
\phi_p(\xp;\bmu - \bmu_b,\bGamma)\label{help2},
\end{align}
\noindent where we have used that $\bDelta^\top\bSigma^{-1}\bDelta = -\blambda^{\scriptscriptstyle\top}\blambda$.
Now, taking the derivative of the ESN density, then
{\small \begin{align}\label{deriv_ESN1}
& -\frac{\partial}{\partial \xp}ESN_p(\xp;\bmu,\bSigma,\blambda,\tau)\nonumber\\
&=
-\xi^{-1}\bigg\{
\frac{\partial}{\partial \xp}
\phi_{p}(\xp;\bmu,\bSigma)
\times
\Phi(\tau+\blambda^{\top}\bSigma^{-1/2}(\mathbf{x}-\bmu))
+
\phi_{p}(\xp;\bmu,\bSigma)
\times
\frac{\partial}{\partial \xp}\Phi_1\big(\tau+\blambda^{\top}\bSigma^{-1/2}(\mathbf{x}-\bmu)\big)\bigg\} \nonumber\\
&= \xi^{-1}\Big\{\bSigma^{-1}
(\xp-\bmu)
\phi_{p}(\xp;\bmu,\bSigma)
\Phi_1(\tau+\blambda^{\top}\bSigma^{-1/2}(\mathbf{x}-\bmu))
-
\bSigma^{-1/2}\blambda\phi_{1}(\tau+\blambda^{\top}\bSigma^{-1/2}(\mathbf{x}-\bmu))
\phi_{p}(\xp;\bmu,\bSigma)\Big\}
\nonumber\\
&= \bigg.\bSigma^{-1}(\xp-\bmu
)ESN_p(\xp;\bmu,\bSigma,\blambda,\tau)
-\xi^{-1}\bSigma^{-1/2}\blambda
\phi_{1}(\tau+\blambda^{\top}\bSigma^{-1/2}(\mathbf{x}-\bmu))
\phi_{p}(\xp;\bmu,\bSigma)
\bigg.\nonumber\\
&\stackrel{\scriptscriptstyle\eqref{help2}}{=} \bSigma^{-1}(\xp-\bmu
)ESN_p(\xp;\bmu,\bSigma,\blambda,\tau)-\eta\bSigma^{-1/2}\blambda\,
\phi_p(\xp;\bmu - \bmu_b,\bGamma)\nonumber\\
&=
\bSigma^{-1}\left\{(\xp-\bmu)ESN_p(\xp;\bmu,\bSigma,\blambda,\tau)-\bdelta\phi_{p}(\xp;\bmu-\bmu_b,\bGamma)\right\},\nonumber
\end{align}}
\noindent with $\eta = {\phi_1(\tau;0,{1+\blambda^\top\blambda})}/{\xi}$ and $\bdelta=\eta\bSigma^{1/2}\blambda$. Multiplying both sides by $\xp^{\bkappa}$ and integrating from $\ap$
to $\bp$, we have (after suppressing the arguments of $\FF_{\bkappa}^{p}$
and $F_{\bkappa}^{p}$) that\\
$$
\bd_{\bkappa}=\bSigma^{-1}\left[\begin{array}{ccccc}
\FF_{\bkappa+\ep_{1}}^{p} & - & \mu_{1}\FF_{\bkappa}^{p} & - & \delta_{1}F_{\bkappa}^{p}\\
\FF_{\bkappa+\ep_{2}}^{p} & - & \mu_{2}\FF_{\bkappa}^{p} & - & \delta_{2}F_{\bkappa}^{p}\\
& \vdots &  & \vdots\\
\FF_{\bkappa+\ep_{p}}^{p} & - & \mu_{p}\FF_{\bkappa}^{p} & - & \delta_{p}F_{\bkappa}^{p}
\end{array}\right],
$$
and the $j$th element of the left hand side is\\
$$
d_{\bkappa,j}=-\int_{\ap(j)}^{\bp(j)}\xp^{\bkappa}ESN_p(\xp;\bmu,\bSigma,\blambda,\tau)\bigg|_{x_{j}=a_{j}}^{b_{j}}\dr \xp(j)+\int_{\ap}^{\bp}k_{j}\xp^{\bkappa-\ep_{j}}ESN_p(\xp;\bmu,\bSigma,\blambda,\tau)\dr \xp
$$
by using integration by parts. Using Proposition \ref{proposition2}, we know that
$$
ESN_p(\xp;\bmu,\bSigma,\blambda,\tau)\big|_{x_{j}=a_{j}}=ESN_1(a_j;\mu_j, \sigma^2_j, c_j\sigma_j\tilde{\varphi}_j,c_j\tau)
ESN_{p-1}(\xp_{(j)};\tilde \bmu^{\ap}_j,\tilde\bSigma_j,\tilde\bSigma^{\scriptscriptstyle 1/2}_j\bvarphi_{(j)},\tilde \tau^{\ap}_j),
$$
$$
ESN_p(\xp;\bmu,\bSigma,\blambda,\tau)\big|_{x_{j}=b_{j}}=ESN_1(b_j;\mu_j, \sigma^2_j, c_j\sigma_j\tilde{\varphi}_j,c_j\tau)
ESN_{p-1}(\xp_{(j)};\tilde \bmu^{\bp}_j,\tilde\bSigma_j,\tilde\bSigma^{\scriptscriptstyle 1/2}_j\bvarphi_{(j)},\tilde \tau^{\bp}_j),
$$
and we obtain
\begin{eqnarray*}
d_{\bkappa,j}&=&k_j \FF ^p_{\bkappa-\ep_j}(\ap,\bp;\bmu, \bSigma,\blambda, \tau )\\
&&+a_j^{k_j} ESN_1(a_j;\mu_j, \sigma^2_j, c_j\sigma_j\tilde{\varphi}_j,c_j\tau) \FF^{p-1}_{\bkappa_{(j)}}(\ap_{(j)},\bp_{(j)};\tilde \bmu^{\ap}_j,\tilde\bSigma_j,\tilde\bSigma^{\scriptscriptstyle 1/2}_j\bvarphi_{(j)},\tilde \tau^{\ap}_j)\nonumber\\
&&-b_j^{k_j} ESN_1(b_j;\mu_j, \sigma^2_j, c_j\sigma_j\tilde{\varphi}_j,c_j\tau) \FF^{p-1}_{\bkappa_{(j)}}(\ap_{(j)},\bp_{(j)};\tilde \bmu^{\bp}_j,\tilde\bSigma_j,\tilde\bSigma^{\scriptscriptstyle 1/2}_j\bvarphi_{(j)},\tilde \tau^{\bp}_j).
\end{eqnarray*}
\noindent Finally, multiplying both sides by
$\bSigma$, we obtain \eqref{ESN_theo_eq1}. This completes the proof.
\end{proof}

This delivers a simple way to compute any arbitrary moments of multivariate TSN distribution $\FF^p_{\bkappa}$ based on at
most $3p + 1$ lower order terms, with $p + 1$ of them being $p$-dimensional integrals, the
rest being $(p-1)$-dimensional integrals, and a normal integral $F^p_{\bkappa}$ that can be easily computed through our proposed  R package \texttt{MomTrunc} available at CRAN. When $k_j=0$, the first term in (\ref{recur1}) vanishes. When $a_j=-\infty$, the second term vanishes, and when $b_j=+\infty$, the third term vanishes. When we have no truncation, that is, all the $a_i's$ are $-\infty$ and all the   $b_i's$ are $+\infty$, for $\Y\sim \ESN_p(\bmu, \bSigma,\blambda, \tau)$, we have
that $$ \FF ^p_{\bkappa}(-\binfty,+\binfty;\bmu, \bSigma,\blambda, \tau )=\mathbb{E}[\Y^{\bkappa}],$$
and in this case the recursive relation is
$$\mathbb{E}[\Y^{\bkappa+\ep_i}]=\mu_i\mathbb{E}[\Y^{\bkappa}]+\delta_i\mathbb{E}[\W^{\bkappa}]+\sum_{j=1}^p \sigma_{ij}k_j \mathbb{E}[\Y^{\bkappa-\ep_i}],\,\,\,i=1,\ldots,p,$$
with $\W\sim N_p(\bmu-\bmu_b,\bGamma)$.

It is worth to stress that any arbitrary truncated moment of $\Y$, that is,
\begin{equation}\label{expect}
\mathbb{E}[\Y^{\bkappa}|\ap\leq\Y\leq\bp] = \frac{\FF_{\bkappa}^p(\ap,\bp;\bmu,\bSigma,\blambda,\tau)}{\LL_p(\ap,\bp;\bmu,\bSigma,\blambda,\tau)},
\end{equation}
can be computed using the recurrence relation given in Theorem \ref{ESN_theo}. In the next section, we proposed another approach to compute \eqref{expect} using a unique corresponding arbitrary moment to a truncated normal vector.

\subsection{Computing ESN moments based on normal moments}


\begin{teorema}\label{theo:FF_to_F}
We have that
\begin{equation*}\label{FF_to_F}
\FF_{\bkappa}^p(\ap,\bp;\bmu,\bSigma,\blambda,\tau) =
\xi^{-1}F_{\bkappa^*}^{p+1}(\ap^{*},\bp^{*};\bmu^*,\bOmega),
\end{equation*}
with $\bmu^*$ and $\bOmega$ as defined in Proposition \ref{propdenscond}, and $\bkappa^{*} = (\bkappa^{\scriptscriptstyle\top},0)^{\top}$, $\ap^{*} = (\ap^{\scriptscriptstyle\top},-\infty)^{\top}$ and $
\bp^{*} = (\bp^{\scriptscriptstyle\top},\tilde\tau)^{\top}$.

In particular, for $\bkappa=\zero$, then
\begin{equation}\label{LL_to_L}
\LL_p(\ap,\bp;\bmu,\bSigma,\blambda,\tau) =
\xi^{-1}L_{p+1}(\ap^{*},\bp^{*};\bmu^*,\bOmega).
\end{equation}
\end{teorema}

\begin{proof}
{
The proof is straightforward by Proposition 3. Since a ESN variate can be written as $\Y \stackrel{d}{=} (\X_1|X_2<\tilde{\tau})$, it follows that
\begin{align*}
\FF_{\bkappa}^p(\ap,\bp;\bmu,\bSigma,\blambda,\tau)
=
\int_{\ap}^{\bp}\y^{\bkappa}
f_\Y(\mathbf{y})\dr\y 
&= 
\frac{1}{P(X_2<\tautil)}
\int_{\ap}^{\bp}\y^{\bkappa}
f_{\X_1}(\y)P(X_2<\tautil|\X_1=\y)
\dr\y\\
&= 
\xi^{-1}
\int_{\ap}^{\bp}
\int_{-\infty}^{\tautil}
\y^{\bkappa}
f_{\X_1}(\y)f_{X_2|\X_1}(x_2|\y)
\dr x_2
\dr\y
\\
&= 
\xi^{-1}
\int_{\ap^*}^{\bp^*}
\xp^{\bkappa^*}
f_{\X}(\x_1,x_2)
\dr \xp
\\
&=
\xi^{-1}F_{\bkappa^*}^{p+1}(\ap^{*},\bp^{*};\bmu^*,\bOmega),
\end{align*}
since $\X = ({\X_1}^\top,X_2)^\top$ is distributed as $N_{p+1}(\bmu^*,\bOmega)$.}

\end{proof}

Equation \eqref{LL_to_L} offers us in a very convenient manner to compute $\LL_p(\ap,\bp;\bmu,\bSigma,\blambda,\tau)$, since efficient algorithms already exist to calculate $L_p(\ap,\bp;\bmu,\bSigma)$ (e.g., see \cite{Genz1999}), which avoids performing $2^p$ evaluations of \emph{cdf} of the multivariate N distribution.

\begin{corollary}\label{cor:expectESN_to_N}
For $\Y \sim ESN_{p}(\bmu,\bSigma,\blambda,\tau)$ and $\X \sim N_{p+1}(\bmu^*,\bOmega)$, it follows from Theorem \ref{theo:FF_to_F} that
$$
\mathbb{E}[\Y^{\bkappa}|\ap \leq \Y \leq \mathbf{b}]=\mathbb{E}[\X^{\bkappa^*}|\ap^* \leq \X \leq \mathbf{b}^*],
$$
with $\ap^{*}$, $
\bp^{*}$, $\bkappa^{*}$, $\bmu^*$ and $\bOmega$ as defined in Theorem \ref{theo:FF_to_F}.
\end{corollary}

\subsection{Mean and covariance matrix of multivariate TESN distributions}\label{subsec:meanvar}

Let us consider $\Y \sim \textrm{TESN}_p(\bmu,\bSigma,\blambda,\tau,[\ap,\bp])$. In light of Theorem \ref{ESN_theo}, we have that
\begin{align*}
\mathbb{E}[Y_i] &= \mu_i+\frac{1}{\LL}
\Bigg[\delta_i L + \sum_{j=1}^p \sigma_{ij}\big[
ESN_1(a_j;\mu_j, \sigma^2_j, c_j\sigma_j\tilde{\varphi}_j,c_j\tau)
\LL_{p-1}(\ap_{(j)},\bp_{(j)};\tilde \bmu^{\ap}_j,\tilde\bSigma_j,\tilde\bSigma^{\scriptscriptstyle 1/2}_j\bvarphi_{(j)},\tilde \tau^{\ap}_j) \nonumber \\
& \;\;\;\;{}-
ESN_1(b_j;\mu_j, \sigma^2_j, c_j\sigma_j\tilde{\varphi}_j,c_j\tau)
\LL_{p-1}(\ap_{(j)},\bp_{(j)};\tilde \bmu^{\bp}_j,\tilde\bSigma_j,\tilde\bSigma^{\scriptscriptstyle 1/2}_j\bvarphi_{(j)},\tilde \tau^{\bp}_j)
\big]\Bigg],
\end{align*}
for $i=1,\ldots, p$, where $\LL\equiv \LL_p(\ap,\bp;\bmu,\bSigma,\blambda,\tau)$ and $L \equiv L_p(\ap,\bp;\bmu-\bmu_b,\bGamma)$.

It follows that
\begin{equation}\label{eq:MEAN_ESN}
\mathbb{E}[\Y] = \bmu+
\frac{1}{\LL}
[
L
\bdelta + \bSigma(\mathbf{q}_a-\mathbf{q}_b)],
\end{equation}
where the $j$-th element of $\mathbf{q}_a$ and $\mathbf{q}_b$ are
\begin{align*}
{q}_{a,j} & = ESN_1(a_j;\mu_j, \sigma^2_j, c_j\sigma_j\tilde{\varphi}_j,c_j\tau)
\LL_{p-1}(\ap_{(j)},\bp_{(j)};\tilde \bmu^{\ap}_j,\tilde\bSigma_j,\tilde\bSigma^{\scriptscriptstyle 1/2}_j\bvarphi_{(j)},\tilde \tau^{\ap}_j) ,\\
{q}_{b,j} & = ESN_1(b_j;\mu_j, \sigma^2_j, c_j\sigma_j\tilde{\varphi}_j,c_j\tau)
\LL_{p-1}(\ap_{(j)},\bp_{(j)};\tilde \bmu^{\bp}_j,\tilde\bSigma_j,\tilde\bSigma^{\scriptscriptstyle 1/2}_j\bvarphi_{(j)},\tilde \tau^{\bp}_j).
\end{align*}

Denoting $\D= [\bd_{\ep_1},\ldots, \bd_{\ep_p}]$, we can write
\begin{align*}
\mathbb{E}[\Y\Y^{\top}] & = \bmu \mathbb{E}[\Y]^{\top}+\frac{1}{\LL}[
L
\bdelta \mathbb{E}[\W]^{\top} + \bSigma\D],\\
\mbox{cov}[\Y] & = \big[\bmu-\mathbb{E}[\Y]\big]\mathbb{E}[\Y]^{\top}+\frac{1}{\LL}[
L
\bdelta \mathbb{E}[\W]^{\top} + \bSigma\D],
\end{align*}

\noindent where $\W \sim \TN_p(\bmu-\bmu_b,\bGamma,[\ap,\bp])$, that is a $p$-variate truncated normal distribution on $[\ap,\bp]$.

Besides, from Corollary \ref{cor:expectESN_to_N}, we have that the first two moments of $\Y$ can be also computed as
\begin{align}
\mathbb{E}[\Y] &= \mathbb{E}[\X]_{(p+1)},\label{eq:MEANVAR_ESN1}\\
\mathbb{E}[\Y\Y^\top] &= \mathbb{E}[\X\X^\top]_{(p+1,p+1)},\label{eq:MEANVAR_ESN2}
\end{align}
with $\X \sim \textrm{TN}_{p+1}(\bmu^*,\bOmega;[\ap^*,\bp^*])$. Note that $\mathrm{cov}[\Y] = \mathbb{E}[\Y\Y^\top] - \mathbb{E}[\Y]\mathbb{E}[\Y^\top]$. Equations \eqref{eq:MEANVAR_ESN1} and \eqref{eq:MEANVAR_ESN2} are more convenient for computing $\EE[\Y]$ and $\mathrm{cov}[\Y]$ since all boils down to compute the mean and the variance-covariance matrix for a $p+1$-variate {TN} distribution which integrals are less complex than the ESN ones.


\subsection{Mean and covariance matrix of TN distributions}\label{subsec:meanvarN}

Some approaches exists to compute the moments of a TN distribution. For instance, for doubly truncation, \cite{bg2009moments} (method available through the \texttt{tmvtnorm R} package) computed the mean and variance of $\X$ directly deriving the MGF of the TN distribution. On the other hand, \cite{kan2017moments} (method available through the \texttt{MomTrunc R} package) is able to compute arbitrary higher order TN moments using a recursive approach as a result of differentiating the multivariate normal density. For right truncation, \cite{vaida2009fast} (see Supplemental Material) proposed a method to compute the mean and variance of $\X$ also by differentiating the MGF, but where the off-diagonal elements of the Hessian matrix are recycled in order to compute its diagonal, leading to a faster algorithm. Next, we present an extension of \cite{vaida2009fast} algorithm to handle doubly truncation.

\subsection{Deriving the first two moments of a double TN distribution through its MGF}\label{proposal}
\bigskip
\begin{teorema}\label{theo:MGF}
Let $\X \sim TN_p(\zero,\R;[\ap,\bp])$, with $\R$ being a correlation matrix of order $p \times p$. Then, the first two moments of $\X$ are given by
\begin{align*}
\mathbbm{E}[\X] &= \left.\frac{\partial m(\tb)}{\partial\tb}\right|_{\tb=\zero}^\top = -\frac{1}{L}\R\q,\\
\mathbbm{E}[\X\X^\top] &= \left.\frac{\partial^2 m(\tb)}{\partial\tb\partial\tb^\top}\right|_{\tb=\zero} = \R + \frac{1}{L}\R\Hb\R,
\end{align*}
and consequently,
\begin{equation*}
\mathrm{cov}[\X] = \R + \frac{1}{L^2}\R\big(L\Hb - \q\q^{\top}\big)\R,
\end{equation*}

\noindent where $L \equiv L_p(\ap,\bp;\zero,\R)$, $\q = \q_a - \q_b$, with the $i$-th element of $\mathbf{q}_a$ and $\mathbf{q}_b$ as
$$
{q}_{a,i} = \phi_1(a_i)
\,L_{p-1}(\ap_{(i)},\bp_{(i)};a_i\R_{(i),i},\tilde\R_i)
\qquad
\text{and}
\qquad
{q}_{b,i} = \phi_1(b_i)
\,L_{p-1}(\ap_{(i)},\bp_{(i)};b_i\R_{(i),i},\tilde\R_i),
$$
$\Hb$ being a symmetric matrix of dimension $p$, with off-diagonal elements $h_{ij}$ given by
\begin{align}\label{hij}
{h}_{ij} &= {h}_{ij}^{aa} - {h}_{ij}^{ba} - {h}_{ij}^{ab} + {h}_{ij}^{bb},\nonumber\\
&= \phi_2(a_i,a_j;\rho_{ij})
\,L_{p-2}(\ap_{(i,j)},\bp_{(i,j)};\bmu_{ij}^{aa},\tilde\R_{ij}) -
\phi_2(b_i,a_j;\rho_{ij})
\,L_{p-2}(\ap_{(i,j)},\bp_{(i,j)};\bmu_{ij}^{ba},\tilde\R_{ij})\nonumber\\
&\quad-
\phi_2(a_i,b_j;\rho_{ij})
\,L_{p-2}(\ap_{(i,j)},\bp_{(i,j)};\bmu_{ij}^{ab},\tilde\R_{ij}) +
\phi_2(b_i,b_j;\rho_{ij})
\,L_{p-2}(\ap_{(i,j)},\bp_{(i,j)};\bmu_{ij}^{bb},\tilde\R_{ij}),\nonumber
\end{align} and diagonal elements
\begin{equation}\label{hii}
h_{ii} = a_i q_{a_i} - b_i q_{b_i} - \R_{i,(i)}\Hb_{(i),i},
\end{equation}
with $\tilde\R_i=\R_{(i),(i)} - \R_{(i),i}\R_{i,(i)}$, $\bmu_{ij}^{\alpha\beta} = \R_{(ij),[i,j]}(\alpha_i,\beta_j)^\top$ and $\tilde\R_{ij}=\R_{(i,j),(i,j)} - \R_{(i,j),[i,j]}\R_{[i,j],(i,j)}$.

\end{teorema}

\noindent{\it Proof.} See Appendix A.

The main difference of our proposal in Theorem \ref{theo:MGF} and other approaches deriving the MGF relies on \eqref{hii}, where the diagonal elements are recycled using the off-diagonal elements $h_{ij}, \ 1\leq i\neq j \leq p$. Furthermore, for $\W\sim TN_p(\bmu,\bSigma;[\tilde\ap,\tilde\bp])$, we have that
\begin{align}
\mathbbm{E}[\W] &= \bmu - \Sb\,\mathbbm{E}[\X],\label{MomMGF1}\\
\mathrm{cov}[\W] &= \Sb\,\mathrm{cov}[\X]\,\Sb,\label{MomMGF2}
\end{align}
where $\bSigma$ being a positive-definite matrix, $\Sb=\mathrm{diag}(\sigma_1,\sigma_2,\ldots,\sigma_p)$, and truncation limits $\tilde\ap$ and $\tilde\bp$ such that $\ap = \Sb^{-1}(\tilde\ap-\bmu)$ and $\bp = \Sb^{-1}(\tilde\bp-\bmu)$.

\section{Dealing with limiting and extreme cases}

Let consider $\Y \sim \textrm{ESN}_p(\bmu,\bSigma,\blambda,\tau)$. As $\tau \rightarrow \infty$, we have that $\xi = \Phi(\tautil) \rightarrow 1$. Besides, as $\tau \rightarrow -\infty$, we have that $\xi \rightarrow 0$ and consequently $\FF_{\bkappa}^p(\ap,\bp;\bmu,\bSigma,\blambda,\tau) =
\xi^{-1}F_{\bkappa^*}^{p+1}(\ap^{*},\bp^{*};\bmu^*,\bOmega) \rightarrow \infty$. Thus, for negative $\tilde\tau$ values small enough, we are not able to compute $\mathbbm{E}[\Y^\kappa]$ due to computation precision. For instance, in \texttt{R} software, $\Phi(\tilde\tau)=0$ for $\tilde\tau < -37$.  The next proposition helps us to circumvent this problem.


\begin{proposition}\label{prop:lambdas} (Limiting distribution for the ESN)
\noindent As $\tau\rightarrow -\infty$,
\begin{equation*}\label{tau_inf_neg0}
ESN_p(\yp;\bmu,\bSigma,\blambda,\btau) {\longrightarrow} \phi_p({\yp};\bmu - \bmu_b,\bGamma).
\end{equation*}
\end{proposition}

\begin{proof}
Let $X_2\sim N(0,1)$. As $\tilde\tau \rightarrow -\infty$, we have that $P(X_2\leq\tilde\tau)\rightarrow 0$, $\mathbbm{E}[X_2|X_2\leq\tilde\tau]\rightarrow\tilde\tau$ and $\mathrm{var}[X_2|X_2\leq\tilde\tau] \rightarrow 0$ (i.e., $X_2$ is (i.e., $\X_2$ is degenerated on $\tautil$). In light of Proposition \ref{stochastic}, $\Y \stackrel{d}{=} (\X_1|X_2=\tilde{\tau})$, and by the conditional distribution of a multivariate normal, it is straightforward to show that $\EE[\X_1|X_2=\tilde{\tau}]=\bmu - \bmu_b$ and $\mathrm{cov}[\X_1|X_2=\tilde{\tau}]=\bGamma$, which concludes the proof.
\end{proof}


\subsection{Approximating the mean and variance-covariance of a TN distribution for extreme cases}

While using the normal relation \eqref{eq:MEANVAR_ESN1} and \eqref{eq:MEANVAR_ESN2}, we may also face numerical problems for extreme settings of $\blambda$ and $\tau$ due to the scale matrix $\bOmega$ does depend on them. Most common problem is that the normalizing constant $L_{p}(\ap^*,\bp^*;\bmu^*,\bOmega)$ is approximately zero, because the probability density has been shifted far from the integration region. It is worth mentioning that, for these cases, it is not even possible to estimate the moments generating Monte Carlo (MC) samples due to the high rejection ratio when subsetting to a small integration region.

For instance, consider a bivariate truncated normal vector $\X=(X_1,X_2)^\top$, with $X_1$ and $X_2$ having zero mean and unit variance, $\mathrm{cov}(X_1,X_2)=-0.5$ and truncation limits $\ap=(-20,-10)^\top$  and $\bp=(-9,10)^ \top$. Then, we have that the limits of $X_1$ are far from the density mass since $P(-20\leq \X_1 \leq -9) \approx 0$. For this case, both the \texttt{mtmvnorm} function from the \texttt{tmvtnorm R} package and the \texttt{Matlab} codes provided in \cite{kan2017moments} return wrong mean values outside the truncation interval $(\ap,\bp)$ and negative variances. Values are quite high too, with mean values greater than $1\times 10^{10}$ and all the elements of the variance-covariance matrix greater than $1\times 10^{20}$. When changing the first upper limit from $-9$ to $-13$, that is $\bp=(-13,10)^\top$, both routines return \texttt{Inf} and \texttt{NaN} values for all the elements.

Although the above scenarios seem unusual, extreme situations that require correction are more common than expected. Actually, the development of this part was motivated as we identified this problem when we fit censored regression models, with high asymmetry and presence of outliers. Hence, we present correction method in order to approximate the mean and the variance-covariance of a multivariate TN distribution even when the numerical precision of the software is a limitation.


\subsubsection*{Dealing with out-of-bounds limits}

Consider the partition $\X = (\X_1^\top,\X_2^\top)^\top$ such that $dim(\X_1) =
p_1$, $dim(\X_2) = p_2$, \CG{where} $p_1 + p_2 = p$.
It is well known that
\begin{equation*}
\EE[\X] = \EE\left[\hspace{-2mm}
\begin{array}{c}
\EE[\X_1|\X_2]
\\
\X_2
\end{array}
\hspace{-2mm}
\right]
\end{equation*}
and
\begin{equation*}
\mathrm{cov}[\X] = \left[\begin{array}{cc}
\EE[\mathrm{cov}[\X_1|\X_2]] + \mathrm{cov}[\EE[\X_1|\X_2]]&
\mathrm{cov}[\EE[\X_1|\X_2],\X_2]\\
\mathrm{cov}[\X_2,\EE[\X_1|\X_2]] &
\mathrm{cov}[\X_2]
\end{array}
\right].
\end{equation*}


Now, consider $\X \sim \mathrm{TN}_{p}\big(\bmu,\bSigma,[\ap,\bp]\big)$ to be partitioned as above. Also consider the corresponding partitions of $\bmu$, $\bSigma$, $\ap = (\ap_1^\top, \ap_2^\top)^\top$ and $\bp = (\bp_1^\top, \bp_2^\top)^\top$. We say that the limits $[\ap_2,\bp_2]$ of $\X_2$ are out-of-bounds if $P(\ap_2\leq \X_2 \leq \bp_2) \approx 0$. Let us consider the case where we are not able to compute any moment of $\X$, because there exists a partition $\X_2$ of $\X$ of dimension $p_2$ that is out-of-bounds. Note this happens because $L_{p}(\ap,\bp;\bmu,\bSigma) \leq P(\ap_2\leq \X_2 \leq \bp_2) \approx 0.$  Also, we consider the partition $\X_1$ such that $P(\ap_1\leq \X_1 \leq \bp_1) > 0$. Since the limits of $\X_2$ are out-of-bounds (and $\ap_2 < \bp_2$), we have two possible cases: $\bp_2 \rightarrow -\binfty$ or $\ap_2 \rightarrow \binfty$. For convenience, let $\bxi_2 = \EE[\X_2]$ and $\bPsi_{22} = \mathrm{cov}[\X_2]$. For the first case, as $\bp_2\rightarrow -\binfty$, we have that $\bxi_2\rightarrow\bp_2$ and $\bPsi_{22}\rightarrow \zero_{p_2\times p_2}$. Analogously, we have that $\bxi_2\rightarrow\ap_2$ and $\bPsi_{22}\rightarrow \zero_{p_2\times p_2}$ as $\ap_2\rightarrow \binfty$.

Then $\X_1 \sim \mathrm{TN}_{p_1}\big(\bmu_1,\bSigma_{11};[\ap_1,\bp_1]\big)$, $\X_2 \sim {N}_{p_2}\big(\bxi_2,\zero\big)$ (i.e., $\X_2$ is degenerated on $\bxi_2$) and $\X_1|\X_2 \sim \mathrm{TN}_{p_1}\big(\bmu_1 + \bSigma_{12}\bSigma_{22}^{-1}(\bxi_2 - \bmu_2),\bSigma_{11} - \bSigma_{12}\bSigma_{22}^{-1}\bSigma_{21};[\ap_1,\bp_1]\big)$. Given that $\mathrm{cov}[\EE[\X_1|\X_2]] = \zero_{p_1\times p_2}$ and
$\mathrm{cov}[\EE[\X_1|\X_2],\X_2] = \zero_{p_2\times p_2}$, it follows that
\begin{equation}\label{cond oob}
\EE[\X] = \left[\hspace{-2mm}
\begin{array}{c}
\bxi_{1.2}
\\
\bxi_2
\end{array}
\hspace{-2mm}
\right]
\qquad \text{and} \qquad
\mathrm{cov}[\X] =
\left[\begin{array}{cc}
\bPsi_{11.2}&
\zero_{p_1\times p_2} \\
\zero_{p_2\times p_1} &
\zero_{p_2\times p_2}
\end{array}
\right],
\end{equation}
with $\bxi_{1.2} = \EE[\X_1|\X_2]$ and $\bPsi_{11.2} = \mathrm{cov}[\X_1|\X_2]$ being the mean and variance-covariance matrix of a TN distribution, which can be computed using \eqref{MomMGF1} and \eqref{MomMGF2}.

In the event that there are double infinite limits, we can partition the vector as well, in order to avoid unnecessary calculation of these integrals.

\subsubsection*{Dealing with double infinite limits}

Let $p_1$ be the number of pairs in $[\ap,\bp]$ that are both infinite. We consider the partition $\X = (\X_1^\top,\X_2^\top)^\top$, such that the upper and lower truncation limits associated with $\X_1$ are both infinite, but at least one of the truncation limits associated with $\X_2$ is finite. Since $\ap_1=-\binfty$ and $\bp_1=\binfty$, it follows that $\X_1 \sim N_{p_1}\big(\bmu_1,\bSigma_{11}\big)$, $\X_2 \sim \mathrm{TN}_{p_2}\big(\bmu_2,\bSigma_{22},[\ap_2,\bp_2]\big)$ and $\X_1|\X_2 \sim N_{p_1}\big(\bmu_1 + \bSigma_{12}\bSigma_{22}^{-1}(\X_2 - \bmu_2),\bSigma_{11} - \bSigma_{12}\bSigma_{22}^{-1}\bSigma_{21}\big)$. This leads to
\begin{align}\label{cond_inf_mean}
\EE[\X] &= \EE\left[\hspace{-2mm}
\begin{array}{c}
\bmu_1 + \bSigma_{12}\bSigma_{22}^{-1}(\X_2 - \bmu_2)
\\
\X_2
\end{array}
\hspace{-2mm}
\right]= \left[\hspace{-2mm}
\begin{array}{c}
\bmu_1 + \bSigma_{12}\bSigma_{22}^{-1}(\bxi_2 - \bmu_2)
\\
\bxi_2
\end{array}
\hspace{-2mm}
\right],
\end{align}
and
\begin{align}\label{cond_inf_var}
\mathrm{cov}[\X] &=
\left[\begin{array}{cc}
\bSigma_{11} -  \bSigma_{12}\bSigma_{22}^{-1}\big(\bI_{p_2} - \bPsi_{22}\bSigma_{22}^{-1}\big)\bSigma_{21}&
\bSigma_{12}\bSigma_{22}^{-1}\bPsi_{22} \\
\bPsi_{22}\bSigma_{22}^{-1}\bSigma_{21} &
\bPsi_{22}
\end{array}
\right],
\end{align}
with $\bxi_2$ and $\bPsi_{22}$ being the mean vector and variance-covariance matrix of a TN distribution,
which can be computed using \eqref{MomMGF1} and \eqref{MomMGF2} as well. 

\indent As can be seen, we can use equations \eqref{cond_inf_mean} and \eqref{cond_inf_var} to deal with double infinite limits, where the truncated moments are computed only over a $p_2$-variate partition, avoiding some unnecessary integrals and saving some computational effort. On the other hand, expression \eqref{cond oob} let us to approximate the mean and the variance-covariance matrix for cases where the computational precision is a limitation.

{
\section{Comparison of computational times}
Since this is the first attempt to compute the moments of a TESN, it is not possible to compare our approach with others methods already implemented in statistical softwares, for instance, R or Stata. However, this section intends to compare three possible approaches to compute the mean vector and variance-covariance matrix of a $p$-variate TESN distribution based on our results.
We consider our first proposal derived from Theorem \ref{ESN_theo} which is derived directly from the ESN pdf, as well as the normal relation given in Theorem \ref{theo:FF_to_F}. For the latter, we use different (some existent) methods for computing the mean and variance-covariance of a TN distribution. The methods that we compare are the following:}

\begin{description}[labelsep=0.5em]
\setlength\itemsep{-0.5em}
\item[Proposal 1:] Theorem \ref{ESN_theo}, i.e., equations \eqref{eq:MEAN_ESN}, and \eqref{eq:MEANVAR_ESN2},
\item[Proposal 2:] Normal relation (NR) in Theorem \ref{theo:FF_to_F} using Theorem \ref{theo:MGF},
\item[{Proposal 3:}] NR in Theorem \ref{theo:FF_to_F} using the \texttt{Matlab} routine from \cite{kan2017moments},
\item[{Proposal 4:}] NR in Theorem \ref{theo:FF_to_F} using the \texttt{tmvtnorm
R} function from \cite{bg2009moments}.
\end{description}
Left panel of Figure 1 shows the number of integrals required to achieve this for different dimensions $p$. We compare the proposal 1 for a $p$-variate TESN distribution and the equivalent $p+1$-variate normal approaches K\&R and proposal 2.

\begin{figure}[h!]
\vspace{7mm}
\includegraphics[width=0.95\textwidth]{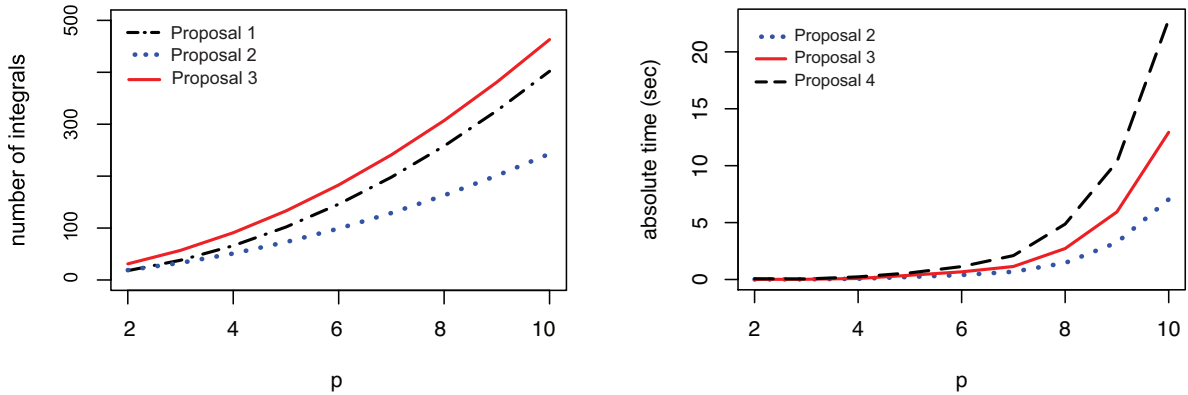}
\caption{Number of integrals required and absolute processing time (in seconds) for computing the mean vector and variance-covariance matrix for a $p$-variate TESN distribution, for 3 different approaches under double truncation.}
\end{figure}

It is clear that the importance of the new proposed method since it reduces the number of integral involved almost to half, this compared to the TESN direct results from proposal 1, when we consider the double truncation. In particular, for left/right truncation, we have that the equivalent $p+1$-variate normal approach along with \cite{vaida2009fast} (now, a special case of proposal 2) requires up to 4 times less integrals than when we use the proposal 3. As seen before, the normal relation proposal 2 outperforms the proposal 1, that is, the equivalent normal approach always resulted faster even it considers one more dimension, that is a $p+1$-variate normal vector, due to its integrals are less complex than for the ESN case.

Processing time when using the equivalent normal approach are depicted in the right panel of Figure 1. Here, we compare the {absolute} processing time of the mean and variance-covariance of a TN distribution under the methods in proposal 2, 3 and 4, for different dimensions $p$. In general, our proposal is the fastest one, as expected. Proposal 3 resulted better only for $p\leq 2$, which confirms the necessity for a faster algorithm, in order to deal with high dimensional problems. Proposal 4 resulted to be the slowest one by far. 
{
\paragraph{Computational time in real life:} For applications where a unique truncated expectation is required (for example, conditional tail expectations as a measure of risk in Finance), the computation cost may seem insignificant, however, iterative algorithms depending on these quantities become computationally intensive. For instance, in longitudinal censored models under a frequentist point of view, an EM algorithm reduces to the computation of the moments of multivariate truncated moments (Lachos et al., 2017) at each iteration, and for all censored observations along subjects. See that, 125K integrals will be required for an algorithm that converges in 250 iterations and a modest dataset with 100 subjects and only four censored observations. Other models as geostatistical models are even more demanding, so small differences in times may be significant between a tractable and non-tractable problem, even that without these expectations, these must be approximated invoking Monte Carlo methods.}

\section{On moments of multivariate folded ESN distributions} \label{sec:folded:ESN}


First, we established some general results for the pdf, cdf and moments of multivariate folded distributions (MFD). These extend the results found in \cite{chakraborty2013multivariate} for a FN distribution to any multivariate distribution, as well as the multivariate location-scale family. The proofs are given in Appendix A.

\begin{teorema}[\textbf{\emph{pdf} and \emph{cdf} of a MFD}]\label{theo2}
Let $\X\in\mathbb{R}^p$ be a $p$-variate random vector with pdf $f_\X(\xp;\btheta)$ and cdf $F_\X(\xp;\btheta)$, with $\btheta$ being a set of parameters characterizing such distribution. If
$\Y=|\X|$, then the joint pdf and cdf of $\Y$ {that follows a
folded distribution of $\X$ are} given, respectively, by
$$f_{\Y}(\yp)=\sum_{\mathbf{s}\in S(p)}f_\X(\bLambda_s\yp;\btheta)
\qquad
\text{and}
\qquad
F_{\Y}(\yp) = \sum_{\mathbf{s}\in S(p)}\pi_s F_\X(\bLambda_s\yp;\btheta)
{,\quad\mbox{for }\yp\geq{\bf  0},}
$$
{where}
$S(p)=\{-1,1\}^p$ is a cartesian product with $2^p$ elements, each of the form $\mathbf{s}=(s_1,\ldots,s_p)$,
$\bLambda_s =
\textrm{Diag}(\mathbf{s})$ and $\pi_{s} = \prod_{i=1}^{p}s_i$.
\end{teorema}

\begin{corollary}\label{cor2}
If $\X \sim f_\X(\xp;\bxi,\bPsi)$ {belongs} to the location-scale family of distributions, with location and scale parameters $\bxi$ and $\bPsi$ respectively, then $\Z_s = \bLambda_s\X \sim f_\X(\zp;\bLambda_s\bxi,\bLambda_s\bPsi\bLambda_s)$ and consequently the joint pdf and cdf of $\Y = |\X|$ {are} given by
$$
f_{\Y}(\yp) = \sum_{\mathbf{s}\in S(p)}f_\X(\yp;\bLambda_s\bxi,\bLambda_s\bPsi\bLambda_s)
\qquad
\text{and}
\qquad
F_{\Y}(\yp) =   \sum_{\mathbf{s}\in S(p)}\pi_s
F_\X(\bLambda_s\yp;\bxi,\bPsi){,\quad\mbox{for }\yp\geq{\bf  0}.}
$$
Hence, the $\bkappa$-th moment of $\Y$ follows as
$$  \mathbb{E}[\Y^{\bkappa}] = \sum_{\mathbf{s}\in S(p)}\mathbb{E}[(\Z_s^{\bkappa})^{\scriptscriptstyle +}],$$
where $\X^{\scriptscriptstyle +}$ denotes the positive component of the random vector $\X$.
\end{corollary}




Let $\X\sim ESN_p(\bmu,\bSigma,\blambda,\tau)$, we now turn our attention to discuss the computation of any arbitrary order moment of $|\X|$, a FESN distribution. Let define the $\II_{\bkappa}^{p} \equiv \II^p_{\bkappa}(\bmu,\bSigma,\blambda,\tau)$ function as
\begin{equation*}\label{eq:I}
\II^p_{\bkappa}(\bmu,\bSigma,\blambda,\tau)=\int_{\textbf{0}}^{\boldsymbol{\infty}}\yp^{\bkappa}ESN_p(\yp;\bmu,\bSigma,\blambda,\tau)\dr\yp.
\end{equation*}
Note that $\II^p_{\bkappa}$ is a special case of $\FF^p_{\bkappa}$ that occurs when $a_i=0$ and $b_i=+\infty$, $\ii$. In this scenario we have
$$\II^{p}_{\bkappa}(\bmu, \bSigma,\blambda, \tau )= \FF ^p_{\bkappa}({\bf 0},+\binf;\bmu, \bSigma,\blambda,\tau).$$
When $\blambda=\zero$ and $\tau=0$, that is, the normal case we write $\II^{p}_{\bkappa}(\bmu, \bSigma,\zero,0) = I^p_{\bkappa}(\bmu,\bSigma)$.

\begin{proposition}\label{cor:folded:ESN}
If $\X\sim \ESN_p(\bmu,\bSigma,\blambda,\tau)$, then $\Z_s = \bLambda_s\X \sim \ESN_p(\bmu_s,\bSigma_s,\blambda_s,\tau)$ and consequently the joint \emph{pdf}, \emph{cdf} and the $\bkappa$th raw moment of $\Y = |\X|$ {are}, respectively, given by
$$f_{\Y}(\yp)=\sum_{\mathbf{s}\in S(p)}ESN_p(\y_p;\bmu_s,\bSigma_s,\blambda_s,\tau),\,\,\,\,
$$
$$
F_{\Y}(\yp) = \LL_p(-\yp,\yp;\bmu,\bSigma,\blambda,\tau),
$$
and
$$  \mathbb{E}[\Y^{\bkappa}] = \sum_{\mathbf{s}\in S(p)}\II^p_{\bkappa}(\bmu_s,\bSigma_s,\blambda_s,\tau),\,\,\,\,\,\,\,$$
where $\y_s={\bLambda_s}\y$, $\bmu_s={\bLambda_s}\bmu$, $\bSigma_s=\bLambda_s\bSigma\bLambda_s$ and $\blambda_s=\bLambda_s\blambda$.
\end{proposition}
\begin{proof}
Note that is suffices to show that,
\begin{center}\emph{if} $\X\sim \ESN_p(\bmu,\bSigma,\blambda,\tau)$, \emph{then} $\Z_s = \bLambda_s\X \sim \ESN_p(\bmu_s,\bSigma_s,\blambda_s,\tau)$,\end{center}
since the rest of the corollary is straightforward. We have that
\small{\begin{align}
ESN_p(\xp;\bmu_s,\bSigma_s,\blambda_s,\tau) &= \xi^{-1}\phi_p(\mathbf{x};\bLambda_s\bmu,\bLambda_s\bSigma\bLambda_s)
\times\Phi_1\big(\tau+(\bLambda_s\blambda)^{\top}(\bLambda_s\bSigma\bLambda_s)^{-1/2}(\mathbf{x}-\bLambda_s\bmu)\big)\nonumber\\
&= \xi^{-1}
|\bLambda_s\bLambda_s|^{1/2}
\phi_p(\bLambda_s^{-1}\mathbf{x};\bmu,\bSigma)\times
\Phi_1\big(\tau+\blambda^{\top}\bLambda_s(\bLambda_s\bSigma\bLambda_s)^{-1/2}\bLambda_s(\bLambda_s^{-1}\mathbf{x}-\bmu)\big)\nonumber\\
&= \xi^{-1}
\phi_p(\bLambda_s\mathbf{x};\bmu,\bSigma) \times\Phi_1\big(\tau+\blambda^{\top}\bLambda_s(\bLambda_s\bSigma\bLambda_s)^{-1/2}\bLambda_s(\bLambda_s\mathbf{x}-\bmu)\big)\label{proofESN1}\\
&\overset{?}{=} \xi^{-1}\phi_p(\bLambda_s\mathbf{x};\bmu,\bSigma)
\times\Phi_1\big(\tau+\blambda^{\top}\bSigma^{-1/2}(\bLambda_s\mathbf{x}-\bmu)\big)\label{proofESN2}\\
&= ESN_p(\bLambda_s\xp;\bmu,\bSigma,\blambda,\tau),\nonumber
\end{align}}
where $\xi^{-1}=\Phi_1\big(\tau/\sqrt{1+\blambda_s^{\scriptscriptstyle\top}\blambda_s}\big)$ due to $\blambda_s^{\top}\blambda_s = \blambda^{\top}\blambda$.

In order to equalize \eqref{proofESN1} and \eqref{proofESN2}, we see that it suffices to show that $\bSigma^{-1/2}=\bLambda_s(\bLambda_s\bSigma\bLambda_s)^{-1/2}\bLambda_s$. This is equivalent to show that $\A=\B$ for $\A=(\bLambda_s\bSigma\bLambda_s)^{1/2}$ and $\B=\bLambda_s\bSigma^{1/2}\bLambda_s$. We have that both matrices $\A$ and $\B$ are positive-definite matrices since $(\bLambda_s\bSigma\bLambda_s)^{1/2}$ and $\bSigma^{1/2}$ are too, as a consequence that they are obtained using Singular Value Decomposition (SVD). Finally, given that $\A^2=\B^2=\bLambda_s\bSigma\bLambda_s$ and any positive-definite matrix has an unique positive-definite square root, we conclude that $\A=\B$ by uniqueness, which concludes the proof.
\end{proof}

\begin{remark}
As a consequence of Proposition \ref{cor:folded:ESN}, we also have the new vectors
${\bdelta}_s = \bLambda_s\bdelta$, ${\bmu_b}_s = \bLambda_s\bmu_b$, ${\bvarphi}_s = \bLambda_s\bvarphi$, ${\tilde\bvarphi}_s = \bLambda_s\tilde\bvarphi$, ${{\tilde\bmu}^{\ap}}_{js} = \bLambda_{s(j)}{\tilde\bmu^{\ap}_j}$ and ${{\tilde\bmu}^{\bp}}_{js} = \bLambda_{s(j)}{\tilde\bmu^{\bp}_j}$, and matrix
$\bGamma_s  = \bLambda_s\bGamma\bLambda_s$, while the constants $\xi$, $\eta$, $c_j$ ,$\tilde\bSigma_j$, and $\tilde\tau_j$ remain invariant with respect to $\bs$.
\end{remark}

From Proposition \ref{cor:folded:ESN}, we can compute any arbitrary moment of a FESN distribution as a sum of $\II_{\bkappa}^{p}$ integrals. In light of Theorem \ref{ESN_theo}, the recurrence relation for $\II_{\bkappa}^{p}$ can be written as
\begin{equation}\label{folded:rec:I:ESN}
\II^{p}_{\bkappa+\ep_i}(\bmu, \bSigma,\blambda, \tau )=\mu_i \II^{p}_{\bkappa}(\bmu, \bSigma,\blambda, \tau )+\delta_i I^{p}_{\bkappa}(\bmu-\bmu_b,\bGamma)+\sum_{j=1}^{p} \sigma_{ij} d_{\kappa,j},\,\,\ii,
\end{equation}
where
\begin{equation*}
d_{\kappa,j}=\left\{\begin{array}{ll}
k_j\II^{p}_{\bkappa-\ep_i}(\bmu, \bSigma,\blambda, \tau ) &;\,\,\,\,\mbox{for}\,\, k_j>0\\
ESN_1(0|\mu_j,\sigma^2_j, c_j\sigma_j\tilde{\varphi}_j,c_j\tau) \II^{p-1}_{\bkappa_{(j)}}(\tilde \bmu_j,\tilde\bSigma_j,\tilde\bSigma^{\scriptscriptstyle 1/2}_j\bvarphi_{(j)},\tilde \tau_j)&;\,\,\,\,\mbox{for}\,\, k_j=0
\end{array}\right.
\end{equation*}
with $\tilde\bmu_j=\bmu_{(j)}-\frac{\mu_j}{\sigma^2_j}\bSigma_{(j)j}$  and $\tilde\tau_j=\tau-\tilde\varphi_j\mu_j$.\\

It is also possible to use the normal relation in Theorem \ref{theo:FF_to_F} to compute $\mathbb{E}[|\X|^{\bkappa}]$ in a simpler manner as in next proposition.

\begin{proposition}\label{folded}
Let $\Y = |\X|$, with $\X \sim ESN_{p}(\bmu,\bSigma,\blambda,\tau)$. In light of Theorem \ref{theo2}, It follows that
\begin{equation*}\label{II to I}
\mathbb{E}[\Y^{\bkappa}] = \xi^{-1}\sum_{\mathbf{s}\in S(p)}I_{\bkappa^*}^{p+1}(\bmu_s^*,\bOmega_s^{-}),
\end{equation*}
where $I^p_{\bkappa}(\bmu,\bSigma) \equiv F^p_{\bkappa}(\zero,\boldsymbol{\infty};\bmu,\bSigma)$,
$\bmu^{*}_s = (\bmu_s^{\scriptscriptstyle\top},\tilde\tau)^{\top}$ and $\bOmega_s = \left(\begin{array}{cc}
\bSigma_s & -\bDelta_s \\
-\bDelta_s^{\top} & 1
\end{array}
\right)$, with $\bmu_s=\CG{\bLambda_s}\bmu$, $\bSigma_s=\bLambda_s\bSigma\bLambda_s$, $\bDelta_s=\bLambda_s\bDelta$ and $\bOmega^{-}_s$ standing for the block matrix $\bOmega_s$ with all its off-diagonal block elements signs changed.
\end{proposition}


Proof is direct from Theorem \ref{theo:FF_to_F} as $\II^p_{\bkappa}$ is a special case of $\FF^p_{\bkappa}$. From Proposition \ref{cor2}, we have that the mean and variance-covariance matrix can be {calculated} as a sum of $2^p$ terms as well, that is
\begin{align}
\label{Imeanvar1}
\mathbb{E}[\Y] &= \sum_{\mathbf{s}\in S(p)}\mathbb{E}[\Z_s^{\scriptscriptstyle +}],\\
\label{Imeanvar2}
\mathrm{cov}[\Y] &= \sum_{\mathbf{s}\in S(p)}\mathbb{E}\hspace{-0.8mm}\left[\Z_s^{\scriptscriptstyle +}{\Z_s^{\scriptscriptstyle +}}^{\scriptscriptstyle\top}\right]
- \mathbb{E}[\Y]\mathbb{E}[\Y]^{\top},
\end{align}
{where $\Z_s^{\scriptscriptstyle +}$ is the positive component of} $\Z_s=\bLambda_s\X\sim \ESN_p(\bmu_s,\bSigma_s,\blambda_s,\tau)$. Note that {there are} $2^p$ times more integrals {to} be calculated {as compared to} the non-folded case, representing a huge computational effort for high dimensional problems.

In order to circumvent this, we can use the fact that $\EE[\Y] = (\EE[Y_1],\ldots,\EE[Y_p])^\top$ and the elements of $\EE[\Y\Y^\top]$ are given by the second moments $\EE[Y_i^2]$ and $\EE[Y_i Y_j], \ 1\leq i \neq j \leq p$. Thus, it is possible to calculate explicit expressions for the mean vector and variance-covariance matrix of the FESN only based on the marginal univariate means and variances of $Y_i$, as well as the covariance terms $\mathrm{cov}(Y_i,Y_j)$.

Next, we circumvent this situation by  propose explicit expressions for the mean and the variance-covariance of the multivariate FESN distribution.

\subsection{Explicit expressions for mean and covariance matrix of multivariate folded
ESN distribution}\label{subsec:explicit}

\noindent
Let $\X \sim ESN_p(\bmu,\bSigma,\blambda,\tau)$.  To obtain
the mean and covariance matrix of $|\X|$ boils down to
compute $\mathbb{E}[|X_i|]$, $\mathbb{E}[|X_i^2|]$ and $\mathbb{E}[|X_i X_j|]$. Consider $X_i$ to be the $i$-th marginal partition of $\X$ distributed as $X_i\sim \ESN(\mu_i,\sigma_i^2,\lambda_i,\tau_i)$. In light of Proposition \ref{cor:folded:ESN} it follows that
\begin{equation*}
\mathbb{E}[|X_i|^k] = \II_k^1(\mu_i,\sigma_i^2,\lambda_i,\tau_i)+\II_k^1(-\mu_i,\sigma_i^2,-\lambda_i,\tau_i).
\end{equation*}
Thus, using the recurrence relation on $\II_k$ in \eqref{folded:rec:I:ESN}, and following the notation in Subsection \ref{univariate_rec}, we can write explicit expressions for $\mathbb{E}[|X_i|]$ and $\mathbb{E}[|X_i|^2]$. High order moments for the univariate FESN and others related distributions are detailed in Appendix  B.

It remains to obtain $\mathbb{E}[|X_iX_j|]$ for $i \neq j$,
which can be obtained as
\begin{align}\label{folded:4terms}
\mathbb{E}[|X_i X_j|] =& \II_{1,1}^{2}(\mu_i,\mu_j,\sigma_i^2,\sigma_{ij},\sigma_j^2,\lambda_i,\lambda_j,\tau)+
\II_{1,1}^{2}(\mu_i,-\mu_j,\sigma_i^2,-\sigma_{ij},\sigma_j^2,\lambda_i,-\lambda_j,\tau)\nonumber\\
&+
\II_{1,1}^{2}(-\mu_i,\mu_j,\sigma_i^2,-\sigma_{ij},\sigma_j^2,-\lambda_i,\lambda_j,\tau)+
\II_{1,1}^{2}(-\mu_i,-\mu_j,\sigma_i^2,\sigma_{ij},\sigma_j^2,-\lambda_i,-\lambda_j,\tau),
\end{align}
as pointed in Proposition \ref{cor:folded:ESN}, with $(X_i,X_j)$ denoting an arbitrary bivariate partition of $\X$. Without loss of generality, let's consider the partition $(X_1,X_2)\sim \ESN_2(\bmu,\bSigma,\blambda,\tau)$ and $(W_1,W_2)\sim N_2(\mathbf{m},\bGamma)$ with $\mathbf{m}=\bmu-\bmu_b$.
For simplicity, we denote $\II_{1,1}^{2} \equiv \II_{1,1}^{2}(\bmu,\bSigma,\blambda,\tau)$, and the normalizing constants
$\mathcal{L}_2 \equiv\mathcal{L}_2(\zero,\binfty;\bmu,\bSigma,\blambda,\tau)$ and $L_2 \equiv {L}_2(\zero,\binfty;\bmu-\bmu_b,\bGamma)$.

Using the recurrence relation on $\II^{2}_{\bkappa+\ep_i}$ in \eqref{folded:rec:I:ESN}, we can obtain $\II_{1,1}^{2}$
for $\bkappa=(1,0)^{\top}$ and $\ep_2=(0,1)^\top$ as
{\small \begin{align}\label{folded:I11:2}
\II_{1,1}^{2}=& (\mu_1\mu_2 + \sigma_{12})\mathcal{L}_2
+ \left(\delta_1\mu_2 + \delta_2(\mu_1 - \mu_{b1})\right){L}_2
+ (\mu_2 \sigma_1^2 + \sigma_{12}) \tilde{\phi}^{(1)}(1-\tilde{\Phi}^{(2.1)}) +
\mu_2\sigma_{12}\tilde{\phi}^{(2)}(1-\tilde{\Phi}^{(1.2)})
\nonumber\\
&+ \delta_2\left[\gamma_1^2\phi(\mu_1;\mu_{b1},\gamma_1^2)(1-\Phi(0;{m}_{2.1},{\gamma}_{2.1}^2))+ \gamma_{12}\phi(\mu_2;\mu_{b2},\gamma_2^2)(1- \Phi(0;{m}_{1.2},{\gamma}_{1.2}^2)))\right]\nonumber\\
&+\sigma_2^2\tilde{\phi}^{(2)}
\II_1^1(\mu_{1.2},\sigma^2_{11.2},\sigma_{11.2}\varphi_1,\tau_{1.2}),\nonumber
\end{align}}
where ${m}_{2.1} = m_2 - \gamma_{12}m_1/\gamma_1^2$,
${m}_{1.2} = m_{1} - \gamma_{12}m_2/\gamma_2^2$,
$\gamma_{2.1}^2 = \gamma_{2}^2 - \gamma_{12}/\gamma_1^2$, $\gamma_{1.2}^2 = \gamma_{1}^2 - \gamma_{12}/\gamma_2^2$,
and in light of Proposition \ref{proposition2}  \ we have that \ 
$\tilde{\Phi}^{(2.1)} \equiv \tilde{\Phi}_1(0;\mu_{2.1},\sigma^2_{2.1},\sigma_{2.1}\varphi_2,\tau_{2.1})$, \ 
$\tilde{\Phi}^{(1.2)} \equiv \tilde{\Phi}_1(0;\mu_{1.2},\sigma^2_{1.2},\sigma_{1.2}\varphi_1,\tau_{1.2})$, and $\tilde{\phi}^{(\ell)} \equiv ESN_1(0;\mu_\ell,\sigma^2_\ell,c_\ell\sigma_\ell\tilde{\varphi}_\ell,c_\ell\tau)$ for $\ell=\{1,2\}$.

Using Remark 1 along with \eqref{folded:4terms}, we finally obtain an explicit expression for $\mathbb{E}[|X_iX_j|]$ as
\begin{align}
\mathbb{E}[|X_iX_j|] =& (\mu_i\mu_j + \sigma_{ij})(1-2(\tilde{\Phi}^{(i)}+\tilde{\Phi}^{(j)}))\nonumber
+ \left(\delta_i\mu_j + \delta_j(\mu_i - \mu_{bi})\right)(1-2({\Phi}^{(i)}+{\Phi}^{(j)}))\\
&+ 2\mu_j\left[\sigma_i^2\tilde{\phi}^{(i)}(1-2\tilde{\Phi}^{(i)}) + \sigma_{ij}\tilde{\phi}^{(j)}(1-2\tilde{\Phi}^{(j)})\right]
\nonumber\\
&+ 2\delta_j\left[\gamma_i^2\phi(\mu_i;\mu_{bi},\gamma_i^2)(1-2\Phi(0;{m}_{j.i},{\gamma}_{j.i}^2))+ \gamma_{ij}\phi(\mu_j;\mu_{bj},\gamma_j^2)(1-2\Phi(0;{m}_{i.j},{\gamma}_{i.j}^2))\right]\nonumber\\
&+2\sigma_j^2\tilde{\phi}^{(j)}\mathbb{E}[|Y_{i.j}|],\nonumber
\end{align}
with $X_{i.j}\sim\ESN_i(\mu_{i.j},\sigma^2_{i.j},\sigma_{i.j}\varphi_i,\tau_{i.j})$. Furthermore,
\begin{equation*}
\tilde{\Phi}^{(1)} \equiv \tilde{\Phi}_2(\zero;(-\mu_i,\mu_j)^\top,\bSigma^{-},(-\lambda_i,\lambda_j)^\top,\tau)
\text{,}\qquad
\tilde{\Phi}^{(2)} \equiv \tilde{\Phi}_2(\zero;(\mu_i,-\mu_j)^\top,\bSigma^{-},(\lambda_i,-\lambda_j)^\top,\tau),
\end{equation*}
\vspace{-0.7cm}
\begin{equation*}
{\Phi}^{(1)} \equiv {\Phi}_2(\zero;(-m_i,m_j)^\top,\bGamma^{-})
\qquad\text{and}\qquad
{\Phi}^{(2)} \equiv {\Phi}_2(\zero;(m_i,-m_j)^\top,\bGamma^{-}),
\end{equation*}

\noindent with $\bSigma^{-}$ ($\bGamma^{-}$) denoting the $\bSigma = [\sigma_{ij}]$ ($\bGamma = [\gamma_{ij}]$) matrix with all its signs of covariances (off-diagonal elements) changed.
Here, we have simplified using the equivalences
\begin{align*}
\mathcal{L}_p(\zero,\binfty;\bmu,\bSigma,\blambda_s,\tau)
&= \tilde{\Phi}_p(\zero;-\bmu_s,\bSigma_s,-\blambda_s,\tau),\qquad\qquad\qquad \text{for }\,\mathbf{s}\in S(p) \\
ESN_p(\zero;\bmu_q,\bSigma_q,\blambda_q,\tau) &= ESN_p(\zero;\bmu_r,\bSigma_r,\blambda_r,\tau),\qquad\qquad\qquad \text{for }\,\mathbf{q},\mathbf{r} \in S(p)\\
P(Y_1Y_2\cdots Y_p>0) &= \sum_{\bs \in S(p)}\pi_s\mathcal{L}_p(\zero,\binfty;\bmu_s,\bSigma_s,\blambda_s,\tau) ,
\end{align*}
with $\pi_{s} = \prod_{i=1}^{p}s_i$ as in Theorem \ref{theo2} and $\sum_{\bs \in S(p)}\mathcal{L}_p(\zero,\binfty;\bmu_s,\bSigma_s,\blambda_s,\tau) = 1$. It is worth mentioning that these expressions hold for the normal case, when $\blambda=\zero$ and $\tau=0$.

As expected, this approach is much faster than the one using equations \eqref{Imeanvar1} and \eqref{Imeanvar2}. For instance, when we consider a trivariate folded ESN distribution, we have that it is approximately 56x times faster than using MC methods and 10x times faster than using equations \eqref{Imeanvar1} and \eqref{Imeanvar2}. Time comparison (summarized in the Figure 
in the Supplementar material, right panel) as well as sample codes of our {\sf MomTrunc R} package are provided in the Appendices C and D, respectively.

\section{Conclusions}\label{Conclusion}
In this paper, we have developed a recurrence approach for computing order product
moments of TESN and FESN distributions as well as explicit expressions for the first two moments as a byproduct, generalizing results obtained by \cite{kan2017moments} for the normal case. The proposed methods also includes the moments of the well-known truncated multivariate SN distribution, introduced by \cite{AzzaliniDV1996}. For the TESN, we have proposed an optimized robust algorithm based only in normal integrals, which for the limiting normal case outperforms the existing popular method for computing the first two moments, even computing these two moments for extreme cases where all available algorithms fail. The proposed method (including its limiting and special cases) has been coded and implemented in the {\sf{R} MomTrunc} package, which is available for {the users} on CRAN repository.

During the last decade or so, censored modeling approaches have been used in
various ways to accommodate increasingly complicated applications. Many of these extensions involve
using Normal (\cite{vaida2009fast}) and Student-t (\cite{Matos.SINICA,lachos2017finite}), however statistical models based on distributions to accommodate
censored and skewness, simultaneously, so far  have remained relatively unexplored in the statistical
literature. We hope that by making the codes available to the community, we will encourage researchers of different fields to use our newly methods. For instance, now it is possible to
derive analytical expressions on the E-step of the EM algorithm for multivariate SN responses with censored observation asblur in \cite{Matos.SINICA}.

Finally, we anticipate in a near future to extend these results to the extended skew-t distribution (\cite{AzzaliniC2003}). We conjecture that our method can be extended to the context of the family of other scale mixtures of skew-normal distributions (\cite{BrancoD2001}). An in-depth investigation of such extension is beyond the scope of the present paper, but it is an interesting topic for further
research.

\begin{center}
{\large\bf SUPPLEMENTARY MATERIAL}
\end{center}

The Supplementary Materials, which is available upon request, contains the following two files:

	\begin{description}
		\item[A] Proofs of propositions and theorems;
		\item[B] Explicit expressions for moments of some folded univariate distributions;
		\item[C] Figures;
		\item[D] The {\sf R MomTrunc} package.
	\end{description}

\bibliographystyle{unsrt}  

\end{document}